\documentclass[11pt,reqno]{amsart}
\usepackage[usenames]{color}
\usepackage{amsmath,pdfsync,verbatim,graphicx,epstopdf,enumerate}
\newcommand{\I}{\mathrm{i}}
\newcommand{\D}{\mathrm{d}}

\newcommand{\lb}{\left(}

\newcommand{\vp}{\varphi}

\newcommand{\rb}{\right)}
\newcommand{\PD}{\partial}

\newcommand{\wt}{\widetilde}

\newcommand{\Dc}{\mathcal{D}}
\newcommand{\Ec}{\mathcal{E}}

\newcommand{\Rb}{\mathbb{R}}

\newcommand{\Beq}{\begin{equation}}
\newcommand{\Eeq}{\end{equation}}
\newcommand{\beq}{\begin{equation*}}
\newcommand{\eeq}{\end{equation*}}
\newcommand{\bal}{\begin{align}}
\newcommand{\eal}{\end{align}}

\newcommand{\g}{\gamma}

\newcommand{\A}{\alpha}
\newcommand{\B}{\beta}
\newcommand{\bp}{\begin{prob}}
\newcommand{\ep}{\end{prob}}
\newcommand{\bpr}{\begin{proof}}
\newcommand{\epr}{\end{proof}}
\renewcommand{\o}{\omega}

\newcommand{\gc}[1]{\begin{quotation}\textbf{Gaik's comment:\
}{\textit{#1}}\end{quotation}}
\newcommand{\rc}[1]{\begin{quotation}\textbf{Raluca's comment:\
}{\textit{#1}}\end{quotation}}

\newcommand{\tc}[1]{\begin{quotation}\textbf{Todd's comment:\
}{\textit{#1}}\end{quotation}}
\newcommand{\vc}[1]{\begin{quotation}\textbf{Venky's comment:\
}{\textit{#1}}\end{quotation}}

\newcommand{\red}{\color{red}}
\newcommand{\tred}[1]{{\color{red}{#1}}}

\newcommand{\st}{\,:\,}

\newcommand{\bel}[1]{\begin{equation}\label{#1}}
\newcommand{\ee}{\end{equation}}

\newcommand{\norm}[1]{|{#1}|}

\newcommand{\rtwo}{\mathbb{R}^2}
\newcommand{\rr}{\mathbb{R}}
\newcommand{\zz}{\mathbb{Z}}
\newcommand{\NT}{\negthinspace}
\newtheorem{theorem}{Theorem}[section]

\newtheorem{lemma}[theorem]{Lemma}
\newtheorem{proposition}[theorem]{Proposition}

\newtheorem{example}[theorem]{Example}

\theoremstyle{definition}
\newtheorem{definition}[theorem]{Definition}
\newtheorem{remark}[theorem]{Remark}

\title[A Class of Singular FIOs in SAR]{A Class of Singular Fourier Integral Operators in Synthetic Aperture Radar Imaging}
\author[Ambartsoumian, Felea, Krishnan, Nolan and Quinto]{G.\ Ambartsoumian$^{\dagger}$, R.\ Felea$^{\diamond}$, V.\ P.\ Krishnan$^{\flat}$, C. \ Nolan$^{\sharp}$ and E.\ T.\ Quinto$^{\ast}$}
\address{$^{\dagger}$Department of Mathematics, University of Texas at Arlington, USA E-mail:{\tt gambarts@uta.edu}
\newline
\indent $^{\diamond}$Department of Mathematics, Rochester Institute of Technology, USA E-mail:{\tt rxfsma@rit.edu}
\newline
\indent $^{\flat}$Tata Institute of Fundamental Research Centre for Applicable Mathematics, Bangalore, India
\newline
\indent\: E-mail:{\tt vkrishnan@math.tifrbng.res.in}
\newline
\indent $^{\sharp}$Department of Mathematics and Statistics, University of Limerick, Ireland E-mail:{\tt clifford.nolan@ul.ie}
\newline
\indent $^{\ast}$Department of Mathematics, Tufts University, USA, E-mail:{\tt todd.quinto@tufts.edu}}
\begin{document}

\begin{abstract}In this article, we analyze the microlocal properties of the linearized
forward scattering operator $F$ and the normal operator
$F^{*}F$ (where $F^{*}$ is the $L^{2}$ adjoint of $F$) which arises in Synthetic Aperture Radar imaging for the
common midpoint acquisition geometry. When $F^{*}$ is applied to the scattered data, artifacts appear. We show that $F^{*}F$
can be decomposed as a sum of four operators, each belonging to a
class of distributions associated to two cleanly intersecting
Lagrangians, $I^{p,l} (\Lambda_0, \Lambda_1)$, thereby explaining the
latter artifacts.
\end{abstract}

\maketitle


\section{Introduction}\label{Section: Intro}

In this article, we analyze the microlocal properties of a transform that appears in Synthetic
Aperture Radar (SAR) imaging. In SAR imaging, a region on the surface of the earth is illuminated by an electromagnetic transmitter and an image of the region is reconstructed based on the measurement of scattered waves at a receiver. For an in-depth treatment of SAR imaging, we refer the reader to \cite{Cheney2001,CheneyBordenbook}.
The transform we study appears as a result of a common midpoint acquisition geometry: the transmitter and receiver move at equal speeds away from a common midpoint along a straight line. This geometry is of interest in bistatic imaging and in certain multiple scattering scenarios \cite{NCDG}.
We first consider the linearized scattering operator $F$ and show that it is an FIO. Since the conventional method of reconstructing the image of an object involves ``backprojecting'' the scattered data, we next study the composition of $F$ with its $L^{2}$ adjoint $F^{*}$.  One of the main goals of this article is to understand the distribution class of $F^{*}F$.

In general the composition of two FIOs is not an FIO. One needs
additional geometric conditions such as the transverse intersection
condition \cite{Ho1971} or the clean intersection condition
\cite{MR0405514} to make the composition operator again an FIO. When
these assumptions fail to be satisfied, it is very useful to study the
canonical relation associated to an FIO by considering the left and
the right projections. More precisely, let $X$ and $Y$ be manifolds
and let $I^m (X, Y; C)$ be the class of Fourier integral operators
(FIOs) $F:\Ec'(X)\to \Dc'(Y)$ of order $m$ associated to the canonical
relation $C \subset (T^*Y \times T^*X) \setminus \{0\}$ and denote by
$\pi_L : C \rightarrow T^*Y, \pi_R : C \rightarrow T^*X$, the left and
right projections respectively. Where and how these projections drop
rank determine the nature of the normal operator $F^*F$.

The singularities which appear in previous work related to SAR
\cite{NC2004, RF1, RF2, Krishnan-Quinto} are folds and blowdowns, that
is, $\pi_L$ and $\pi_R$ have both fold singularities or $\pi_L$ has a
fold singularity and $\pi_R$ has a blowdown singularity. These
singularities will be defined in Section \ref{Section: Prelims}. Then it is known that the corresponding normal operator belongs to a class of distributions  $I^{2m,0}
(\Delta, \tilde{C})$ introduced by \cite{Guillemin-Uhlmann} (and defined in Section \ref{Section: Prelims}). This means that the adjoint operator $F^*$ introduces an additional
singularity given by $\tilde{C}$ apart from the  initial one given by
$\Delta$. 
In
this article, the linearized scattering operator $F$
exhibits a new feature: both
projections drop rank by one on a disjoint union of two smooth hypersurfaces
$\Sigma_1 \cup \Sigma_2$. On each of them, $\pi_L$ is a
projection with fold singularities and $\pi_R$ is a projection with blowdown singularities.  We then show
that $F^*F$ belongs to the class
$I^{2m,0}(\Delta,C_1)+I^{2m,0}(\Delta, C_2)+
I^{2m,0}(C_{1},C_{3})+I^{2m,0}(C_{2},C_{3})$ (where these classes are
given in Definition \ref{def:Ipl}).  This means that the adjoint operator $F^*$ adds three
more singularities given by $C_1, C_2, C_3$ in addition to the true reconstructed singularity given by $\Delta$. We clarify this in detail in Section \ref{Section: Analysis}. The main tool for proving our result is the iterated
regularity property; a characteristic property of $I^{p,l}$ classes \cite[Proposition 1.35]{GU1990a}.

\section{Statement of the main results}\label{Section: Statements}
\subsection{The linearized scattering model}
For simplicity, we assume that both the transmitter and receiver are at the same height $h>0$ above the ground at all times and move in opposite directions at equal speeds along the line parallel to the $x$ axis from the common midpoint $(0,0,h)$. Such a model arises when considering signals which have scattered from a wall within the vicinity of
a scatterer and can be understood in the context of the method of images; see \cite{NCDG} for more details.

Let $\g_{T}(s)=(s,0,h)$ and $\g_{R}(s)=(-s,0,h)$ for $s\in (0,\infty)$ be the trajectories of the transmitter and receiver  respectively.

The linearized model for the scattered signal we will use in this article is \cite{NCDG}
\begin{equation}\label{Bi-static Mathematical Model}
d(s,t):=F V(s,t)=\int e^{-\I \omega(t-\frac{1}{c_{0}}R(s,x))}a(s,x,\omega)V(x)\D x\D \omega \mbox{ for } (s,t)\in (0,\infty)\times (0,\infty),
\end{equation}
where $V(x)=V(x_{1},x_{2})$ is the function modeling the object on the ground, $R(s,x)$ is the bistatic distance:

\[
R(s,x)=|\g_{T}(s)-x|+|x-\g_{R}(s)|,
\]
 $c_{0}$ is the speed of electromagnetic wave in free-space and the amplitude term $a$ is given by
\begin{equation}\label{Bi-static Amplitude}
a(s,x,\omega)=\frac{\omega^{2}p(\omega)}{16\pi^{2}|\g_{T}(s)-x||\g_{R}(s)-x|},
\end{equation}
where $p$ is the Fourier transform of the transmitted waveform.

\subsection{Preliminary modifications on the scattered data} For
simplicity, from now on we will assume that $c_{0}=1$.  To make the
composition of $F$ with its $L^{2}$ adjoint $F^{*}$ to be
well-defined, we multiply $d(s,t)$ by an infinitely differentiable
function $f(s,t)$ identically $1$ in a compact subset of
$(0,\infty)\times (0,\infty)$ and supported in a slightly bigger
compact subset of $(0,\infty)\times (0,\infty)$. We rename $f\cdot d$
as $d$ again.

As we will see below, our method cannot image a neighborhood of the
common midpoint. That is, if the transmitter and receiver are at
$(s,0,h)$ and $(-s,0,h)$ respectively, we cannot image a neighborhood
of the origin on the plane. Therefore we modify $d$ further by
considering a smooth function $g(s,t)$ such that \Beq\label{The
function g(s,t)} g(s,t)=0 \mbox{ for } (s,t):
|t-2\sqrt{s^{2}+h^{2}}|<20\epsilon^{2}/h, \Eeq where $\epsilon>0$ is
given. Again we let $g\cdot d$ to be $d$ and $g \cdot a$ to be $a$.
The choice of constant  on the right hand side of \eqref{The function g(s,t)} will be justified in Appendix \ref{Factor in g(s,t)}.

\Beq\label{Modified Bi-static Mathematical Model} d(s,t)=FV(s,t)=\int
e^{-\I\omega\left(t-\sqrt{(x_1-s)^2+x_2^2+h^2}-\sqrt{(x_1+s)^2+x_2^2+h^2}\
\right)} a(s,t,x,\omega) V(x)\:\D x \D \o  \Eeq
and the phase function is \bel{def:vp}
\vp(s,t,x,\omega)
=\omega\left(t-\sqrt{(x_1-s)^2+x_2^2+h^2}-\sqrt{(x_1+s)^2+x_2^2+h^2}\
\right).\ee
From now on, we will denote
the $(s,t)$ space as $Y$ and $x=(x_{1},x_{2})$ space as $X$.

We assume that the amplitude function $a \in S^{m+\frac{1}{2}}$, that is, it
satisfies the following estimate: For every compact set $K\subset
Y\times X$, non-negative integer $\A$, and $2$-indexes
$\beta=(\beta_{1},\beta_{2})$ and $\g$, there is a constant $c$ such
that
\begin{equation}\label{Amplitude Estimate}
|\PD_{\omega}^{\A}\PD_{s}^{\B_{1}}\PD_{t}^{\B_{2}}\PD_{x}^{\g}a(s,t,x,\omega)|\leq c(1+|\omega|)^{m+(1/2)-\A}.
\end{equation}
This assumption is satisfied if the transmitted waveform from the antenna is approximately a Dirac delta distribution.

With these modifications, we show that $F$ is a Fourier integral
operator of order $m$ and study the properties of the natural
projection maps from the canonical relation of $F$.  Our first main
result is the following:
\begin{theorem}\label{FIOTheorem} Let $F$ be as in \eqref{Modified Bi-static Mathematical Model}. Then
\begin{enumerate}[(a)]
\item $F$ is an FIO of order $m$.
\item The canonical relation $C$ associated to $F$ is given by
\begin{align}
\notag C =\Bigg{\{}\Bigg{(}& s,t,-\omega\Big{(}\frac{x_1-s}{\sqrt{(x_1-s)^2+x_2^2+h^2}}-\frac{x_1+s}{\sqrt{(x_1+s)^2+x_2^2+h^2}}\Big{)}, -\omega;\\
&x_1,x_2,-\omega\Big{(}\frac{x_1-s}{\sqrt{(x_1-s)^2+x_2^2+h^2}}+\frac{x_1+s}{\sqrt{(x_1+s)^2+x_2^2+h^2}}\Big{)},\\
\notag
&-\omega\Big{(}\frac{x_2}{\sqrt{(x_1-s)^2+x_2^2+h^2}}+\frac{x_2}{\sqrt{(x_1+s)^2+x_2^2+h^2}}\Big{)}\Bigg{)}\mbox{
where}\\
\notag & s>0,\ t=\sqrt{(x_1-s)^2+x_2^2+h^2}+
\sqrt{(x_1+s)^2+x_2^2+h^2}, x\neq 0, \mbox{ and } \omega\neq 0\Bigg{\}},
\end{align}
 and $C$ has global parameterization $ (0,\infty)\times
 \left(\rtwo\setminus{0}\right)\times\left(\rr\setminus{0}\right)\ni(s,x_1,x_2,\omega)\mapsto
 C$.

  \item Let $\pi_{L}:C\to T^{*}Y$ and $\pi_{R}:C\to T^{*}X$ be the
left and right projections respectively. Then $\pi_{L}$ and $\pi_{R}$
drop rank simply by one on a set $\Sigma =\Sigma_{1}\cup
\Sigma_{2}$ where in the coordinates $(s,x,\omega)$, $\Sigma_1=\{ (s,x_{1},0,\omega) | s>0, |x_{1}|> \epsilon',\omega\neq
0\}$ and $\Sigma_2=\{ (s,0,x_{2}, \omega) |s>0, |x_2|> \epsilon',\omega\neq 0\}$ for $0<\epsilon'$ small enough.
\item   $ \pi_L$ has  a fold singularity along $\Sigma$.  \item $
\pi_{R}$ has a blowdown singularity along $\Sigma$.
\end{enumerate}
\end{theorem}

\begin{remark}
Note that due to the function $g(s,t)$ of \eqref{The function g(s,t)}
in the amplitude, it is enough to consider only points in $C$ that are
strictly away from $\{(s,0,\omega): s>0, \omega\neq 0\}$. This is reflected in
the definitions of $\Sigma_{1}$ and $\Sigma_{2}$, where $|x_{1}|$ and
$|x_{2}|$, respectively, are strictly positive.
\end{remark}
\begin{remark} Note that $C$ is even with respect to both $x_{1}$ and $x_{2}$. In other words $C$ is a 4-1 relation. This observation suggests that $\pi_{L}$ (respectively $\pi_{R}$) has two fold (respectively blowdown) sets. See Proposition \ref{Fold/blowdown proposition}.
\end{remark}

We then analyze the normal operator $F^{*}F$. Our next main result is
the following: \begin{theorem} Let $F$ be as in \eqref{Modified
Bi-static Mathematical Model} of order m. Then $F^*F$ can be
decomposed into a sum belonging to
$I^{2m,0}(\Delta,C_1)+I^{2m,0}(\Delta, C_2)+
I^{2m,0}(C_{1},C_{3})+I^{2m,0}(C_{2},C_{3})$ where these classes
are given in Definition \ref{def:Ipl}.
\end{theorem}

In Remark \ref{remark:strength}, we will explain why the added
singularities given by $C_1, C_2, C_3$ have the same strength as the
object singularities given by $\Delta$.


\section{Preliminaries} \label{Section: Prelims}

\subsection{Singularities and $I^{p,l}$ classes}\label{sect:Ipl}

\vskip .5 cm

\par


In this section we will define  fold and blowdown
singularities and describe the $I^{p,l}$
class of distributions required for the analysis of the composition
operator $F^{*}F$.
 \begin{definition}[{\cite[p.109-111]{MR999388}}]\label{def:fold-blowdown}  Let $M$
and $N$ be manifolds of dimension $n$ and let $f:M\to N$ be
$C^\infty$. Let $\Omega$ be a non-vanishing volume form on $N$ and define $\Sigma= \{\sigma \in M\st f^*\Omega(\sigma) = 0\}$, that is,  $\Sigma$ is the set of critical points of $f$. Note that, equivalently, $\Sigma$ is defined by the vanishing of the determinant of the Jacobian of $f$.
\begin{enumerate}[(a)]
\item If for all $\sigma\in\Sigma$, we have (i) the corank of $f$ at $\sigma$ is $1$, (ii) $\ker(\D f_\sigma)\cap T_\sigma\Sigma=\{0\}$, (iii) $f^*\Omega$ vanishes exactly to first order on $\Sigma$, then we say that $f$ is a {\em fold}.
\item If for all $\sigma\in\Sigma$, we have (i)  the rank of $f$ is constant; let us call this constant $k$, (ii) $\ker(\D f_\sigma)\subset T_\sigma\Sigma$, (iii) $f^*\Omega$ vanishes exactly to order $n-k$ on $\Sigma$, then we say that $f$ is a {\em blowdown}.
\end{enumerate}
\end{definition}

We now define $I^{p,l}$ classes. They were first introduced by Melrose
and Uhlmann \cite{MU}, Guillemin and Uhlmann \cite{Guillemin-Uhlmann} and Greenleaf
and Uhlmann \cite{GU1990a} and they were used in the context of radar
imaging in \cite{NC2004,RF1, RF2}.

\begin{definition} Two submanifolds  $M$ and $N$ intersect {\it cleanly} if $M \cap N$ is a smooth submanifold and $T(M \cap N)=TM \cap TN$.
\end{definition}

 Let us consider the following example

\begin{example}\label{ex:model}  Let
$\tilde{\Lambda}_0=\Delta_{T^*\rr^n}=\{ (x, \xi; x,
\xi)| x \in \rr^n, \ \xi \in \rr^n \setminus 0 \}$ be  the diagonal in
$T^*\rr^n \times T^*\rr^n$ and let  $\tilde{\Lambda}_1= \{ (x', x_n, \xi', 0; x',
y_n, \xi', 0)| x' \in \rr^{n-1}, \ \xi' \in \rr^{n-1} \setminus 0 \}$.
Then, $\tilde{\Lambda}_0$ intersects $ \tilde{\Lambda}_1$ cleanly in
codimension $1$.\end{example}

Now we define the class of product-type symbols $S^{p,l}(m,n,k)$.

\begin{definition}
$S^{p,l}(m,n,k)$ is the set of all functions $a(z,\xi,\sigma) \in
C^{\infty} ({\rr}^m \times {\rr}^n \times {\rr}^k )$ such that
for every $K \subset {\rr}^m$ and every $\alpha \in {\zz}^m_+,
\beta \in {\zz}^n_+, \gamma \in {\zz}^k_+$ there is $c_{K, \alpha,
\beta,\gamma}$ such
that
\[|\partial_z^{\alpha}\partial_{\xi}^{\beta}\partial_{\sigma}^{\gamma}
a(z,\xi,\sigma)| \le c_{K,\alpha,\beta,\gamma}(1+ |\xi|)^{p- |\beta|} (1+
|\sigma|)^{l-| \gamma|}, \forall (z,\xi,\tau) \in K \times {\rr}^n
\times {\rr}^k.\]
\end{definition}

Since any two sets of cleanly intersecting Lagrangians are equivalent
\cite{Guillemin-Uhlmann}, we first define $I^{p,l}$ classes for the
case in Example \ref{ex:model}.

\begin{definition}\cite{Guillemin-Uhlmann}    Let $I^{p,l}(\tilde{\Lambda}_0,
\tilde{\Lambda}_1)$  be the set of all distributions  $u$
such that $u=u_1 + u_2$  with $u_1 \in C^{\infty}_0$
 and $$u_2(x,y)=\int e^{i((x'-y')\cdot \xi'+(x_n-y_n-s) \cdot \xi_n+ s
\cdot \sigma)} a(x,y,s; \xi,\sigma)d\xi d\sigma ds$$  with $a \in
S^{p',l'}$  where $p'=p-\frac{n}{2}+\frac{1}{2}$  and
$l'=l-\frac{1}{2}$.
\end{definition}

Let $(\Lambda_0, \Lambda_1)$ be a pair of cleanly intersection
Lagrangians in codimension $1$ and let $\chi$ be a canonical
transformation which maps $(\Lambda_0, \Lambda_1)$ into
$(\tilde{\Lambda}_0, \tilde{\Lambda}_1)$ and maps $\Lambda_0\cap
\Lambda _1$ to $\tilde{\Lambda}_0\cap \tilde{\Lambda}_1$, where
$\tilde{\Lambda}_j$ are from Example \ref{ex:model}. Next we define
the $I^{p,l}(\Lambda_0, \Lambda_1)$.

\begin{definition}[\cite{Guillemin-Uhlmann}]\label{def:Ipl}
Let  $I^{p,l}( \Lambda_0, \Lambda_1)$
be the set of all distributions $u$ such that $ u=u_1 + u_2 + \sum
v_i$ where $u_1 \in I^{p+l}(\Lambda_0 \setminus \Lambda_1)$, $u_2 \in
I^{p}(\Lambda_1 \setminus \Lambda_0)$, the sum $\sum v_i$ is locally
finite and $v_i=Aw_i$ where $A$ is a zero order FIO associated to
$\chi ^{-1}$, the canonical transformation from above, and $w_i \in
I^{p,l}(\tilde {\Lambda}_0, \tilde{\Lambda}_1)$.

If $u$ is the Schwartz kernel of the linear operator $F$, then we say
$F\in I^{p,l}(\Lambda_0,\Lambda_1)$.
\end{definition}
This class of distributions is invariant under FIOs associated to
canonical transformations which map the pair $(\Lambda_0, \Lambda_1)$
to itself.   If
$F \in I^{p,l}(\Lambda_0, \Lambda_1)$ then $F \in I^{p+l}(\Lambda_0
\setminus \Lambda_1)$ and $F \in I^p(\Lambda_1 \setminus \Lambda_0)$
\cite{Guillemin-Uhlmann}.  Here by $F\in I^{p+l}(\Lambda_{0} \setminus
\Lambda_{1})$, we mean that the Schwartz kernel of $F$ belongs to
$I^{p+l}(\Lambda_{0}\setminus \Lambda_{1})$ microlocally away from
$\Lambda_{1}$. \\
One way to show that a distribution belongs to $I^{p,l}$ class is by
using the iterated regularity property:

\begin{proposition} \cite[Proposition 1.35]{GU1990a}\label{Ipl3}
If $u \in \Dc'(X \times Y)$ then $u \in I^{p,l} (\Lambda_0,
\Lambda_1)$ if there is an $s_0 \in R$ such that for all first order
pseudodifferential operators $P_i$ with principal symbols vanishing on
$ \Lambda_0 \cup \Lambda_1$, we have $P_1P_2 \dots P_r u \in
H^{s_0}_{loc}$.
\end{proposition}


\section{Analysis of the Operator $F$}\label{Section: Forward Operator}
In this Section, we prove Theorem \ref{FIOTheorem}, as  a result  of Lemma \ref{Lemma:FIO}
and Proposition \ref{Fold/blowdown proposition}.
\begin{lemma}\label{Lemma:FIO}
$F$ is an FIO of order $m$ with the canonical relation $C$  given by
\begin{align}\label{Common-midpoint Canonical Relation}
\notag C =\Bigg{\{}\Bigg{(}& s,t,-\omega\Big{(}\frac{x_1-s}{\sqrt{(x_1-s)^2+x_2^2+h^2}}-\frac{x_1+s}{\sqrt{(x_1+s)^2+x_2^2+h^2}}\Big{)}, -\omega;\\
&x_1,x_2,-\omega\Big{(}\frac{x_1-s}{\sqrt{(x_1-s)^2+x_2^2+h^2}}+\frac{x_1+s}{\sqrt{(x_1+s)^2+x_2^2+h^2}}\Big{)},\\
\notag &-\omega\Big{(}\frac{x_2}{\sqrt{(x_1-s)^2+x_2^2+h^2}}+\frac{x_2}{\sqrt{(x_1+s)^2+x_2^2+h^2}}\Big{)}\Bigg{)}\mbox{ with}\\
\notag &s>0, t=\sqrt{(x_1-s)^2+x_2^2+h^2}+ \sqrt{(x_1+s)^2+x_2^2+h^2},
x\in \rr^2\setminus \{0\}, \omega\neq 0\Bigg{\}}.
\end{align}
We note that $ (0,\infty)\times (\rtwo\setminus{0})\times (\rr\setminus 0)\ni
(s,x_{1},x_{2},\omega)\mapsto C$ is a global parametrization of $C$.
\end{lemma}

We will use the coordinates $(s,x,\omega)$ in this lemma from now on to
describe $C$ and subsets of $C$.

\bpr The phase function $\vp$ is non-degenerate with $\PD_{x}\vp$,
$\PD_{s,t}\vp$ nowhere $0$ whenever $\PD_\omega\vp=0$. We should mention that $\nabla \PD_{\omega}\vp \neq 0 $.
(Note that in order for $\PD_{x}\vp$ to be
nowhere $0$, we require exclusion of the common midpoint from our
analysis).  This observation is needed to show $F$ is a FIO
rather than just a Fourier integral distribution. Recalling that $a$
satisfies amplitude estimates (\ref{Amplitude Estimate}), we conclude
that $F$ is an FIO \cite{FT2}.  Also since $a$ is of order
$m+\frac{1}{2}$, the order of the FIO is $m$ \cite[Definition
3.2.2]{DuistermaatBook}.  By definition \cite[Equation
(3.1.2)]{Ho1971}
\[
C = \{((s,t,\PD_{s}\vp,\PD_{t}\vp);(x,-\PD_{x}\vp)) \st
\PD_{\omega}\vp=0\}.\] A calculation using this
definition establishes \eqref{Common-midpoint Canonical Relation}.
Furthermore, it is easy to see that $(s,x_{1},x_{2},\omega)$ is a global
parametrization of $\Lambda$.  \epr

\begin{remark} In the SAR application, $a$ has order 2 which makes
operator $F$ of order $\frac{3}{2}$. But from now on will consider
that $F$ has order $m$.\end{remark}

\begin{proposition} \label{Fold/blowdown proposition}
Denoting the restriction of the left and right projections to $C$ by $\pi_{L}$ and $\pi_{R}$ respectively, we have
\begin{enumerate}[(a)]
\item $\pi_{L}$ and $\pi_{R}$ drop rank by one on a set $\Sigma
=\Sigma_{1}\cup \Sigma_{2}$.
 Here we use the global
coordinates from Lemma \ref{Lemma:FIO}.  \item  $ \pi_L$ has  a fold singularity along $\Sigma$.  \item  $ \pi_{R} $ has  a blowdown singularity along $\Sigma$.
\end{enumerate}
\end{proposition}

\begin{proof}
Let $A=\sqrt{(x_1-s)^2+x_2^2+h^2}$ and $B=\sqrt{(x_1+s)^2+x_2^2+h^2}$. We have
\[
\pi_L(x_1, x_2, s, \omega)=(s,A+B,-(\frac{x_1-s}{A}-\frac{x_1+s}{B} )\omega, -\omega)
\]
 and
\[{\D \pi_L}= \left(\begin{matrix}
0 & 0 & 1 & 0\\
\frac{x_1-s}{A} + \frac{x_1+s}{B}& \frac{x_2}{A} + \frac{x_2}{B} & \ast & 0 \\
-\omega(\frac{x_2^2+h^2}{A^3} - \frac{x_2^2+h^2}{B^3}) & \omega(\frac{(x_1-s)x_2}{A^3} - \frac{(x_1+s)x_2}{B^3}) & \ast & \ast\\
0 & 0 & 0 & -1
\end{matrix} \right)
\] and \bel{det:dpiL}\det \D \pi_L=\frac{4x_1x_2s\omega}{A^2B^2}
(1+\frac{(x_1^2-s^2+x_2^2+h^2}{AB})\ee
 We have that $s>0$ and the number in the parenthesis is a positive
 number by Lemma
\ref{Non-vanishing Lemma}  below.

Therefore, $\pi_L$ drops rank by one on $\Sigma=\Sigma_1 \cup
\Sigma_2$.  To show $d(\det(\D \pi_L))$ is nowhere zero on
$\Sigma$, one uses the product rule in \eqref{det:dpiL} and the fact
that the differential of $\frac{4x_1x_2s\omega}{A^2B^2}$ is never zero
on $\Sigma$ and the inequality in Lemma \ref{Non-vanishing Lemma}.

On $\Sigma_1$ the kernel of $\D \pi_L$ is $\frac{\partial}{\partial
x_2}$ which is transversal to $\Sigma_1$ and on $\Sigma_2$ the kernel
of $\D\pi_L$ is $\frac{\partial}{\partial x_1}$ which is transversal
to $\Sigma_2$. This means that $\pi_L$ has a fold singularity along  $\Sigma$.

 \par Similarly,
\[
\pi_R(x_1, x_2, s, \omega)=(x_1,x_2,-(\frac{x_1-s}{A}+\frac{x_1+s}{B} )\omega, -(\frac{x_2}{A}+\frac{x_2}{B}) \omega).
\]
Then
\[
{\D \pi_R}= \left(\begin{matrix}
1 & 0 & 0 & 0\\
0 & 1 & 0 & 0\\
\ast & \ast & \omega(\frac{x_2^2+h^2}{A^3} - \frac{x_2^2+h^2}{B^3})& -(\frac{x_1-s}{A} + \frac{x_1+s}{B})\\
\ast & \ast & -\omega(\frac{(x_1-s)x_2}{A^3} - \frac{(x_1+s)x_2}{B^3}) & -(\frac{x_2}{A} + \frac{x_2}{B})
\end{matrix} \right)
\]
has the same determinant so $\pi_R$ drops rank by one on $\Sigma$ and
the kernel of $\D\pi_R$ is a linear combination of
$\frac{\partial}{\partial \omega}$ and $\frac{\partial}{\partial s}$
which are tangent to both $\Sigma_1 $ and $\Sigma_2$.  This
means that $\pi_R$ has  a blowdown singularity along  $\Sigma$.

\end{proof}

\begin{lemma}\label{Non-vanishing Lemma}
For all $s\neq 0$,
\[
1+\frac{x_{1}^{2}-s^{2}+x_{2}^{2}+h^{2}}{|x-\g_{T}(s)||x-\g_{R}(s)|}>0.
\]
\end{lemma}
\bpr
Equivalently, we show that $(|x-\g_{T}(s)||x-\g_{R}(s)|)^{2}>(x_{1}^{2}+x_{2}^{2}+h^{2}-s^{2})^{2}$. Expanding out both sides and simplifying, we obtain $4s^{2}(x_{2}^{2}+h^{2})>0$ which holds for $s\neq 0$, since $h>0$. Therefore the lemma is proved.
\epr

\section{Analysis of the normal operator $F^*F$} \label{Section: Analysis}

We have
\begin{align*}
F^{*}FV(x)=\int & e^{\I \omega(t-(|x-\g_{T}(s)|+|x-\g_{R}(s)|))-\wt{\omega}(t-(|y-\g_{T}(s)|+|y-\g_{R}(s)|))}\\
&\times \overline{a(s,t,x,\omega)}a(s,t,y,\wt{\omega})V(y)\D s \D t\D \omega\D \wt{\omega}\D y.
\end{align*}
After an application of the method of stationary phase in $t$ and
$\wt{\omega}$, the Schwartz kernel of this operator is
\Beq\label{Kernel of composed operator} K(x,y)=\int e^{\I \omega\lb
|y-\g_{T}(s)|+|y-\g_{R}(s)|-(|x-\g_{T}(s)|+|x-\g_{R}(s)|)\rb}\wt{a}(x,y,s,\omega)\:
\D s \D \omega.  \Eeq Note that $\wt{a}\in S^{2m+1}$ since we assume $a\in
S^{m+1/2}$.

Let the phase function of the kernel $K$ be denoted by
\Beq\label{Phase of kernel}
\Phi=\omega\lb  |y-\g_{T}(s)|+|y-\g_{R}(s)|-(|x-\g_{T}(s)|+|x-\g_{R}(s)|)\rb.
\Eeq
\begin{proposition}\label{Wavefront of composition} The wavefront set of the kernel $K$ of $F^{*}F$ satisfies,
\[
WF(K)'\subset \Delta \cup C_{1}\cup C_{2}\cup C_{3},
\]
where $\Delta$ is the diagonal  in $T^*X \times T^*X$ and the Lagrangians $C_{i}$ for $i=1,2,3$ are the graphs of the following functions $\chi_{i}$ for $i=1,2,3$ on $T^{*}X$:
 \[
 \chi_1(x,\xi)=(x_1, -x_2,  \xi_1, -\xi_2), \chi_{2}(x,\xi)=(-x_{1},x_{2},-\xi_{1},\xi_{2}) \mbox{ and } \chi_{3}=\chi_{1}\circ \chi_{2}.
 \]
 Furthermore we have:
 \begin{enumerate}[(a)]
 \item $\Delta$ and $C_{1}$, $\Delta$ and $C_{2}$, $C_{1}$ and $C_{3}$, $C_{2}$ and $C_{3}$  intersect cleanly in codimension 2.
  \item $\Delta\cap C_{3}=C_{1}\cap C_{2}=\emptyset$
  \end{enumerate}
 \end{proposition}
\bpr
In order to find the wavefront set of the kernel $K$, we consider the canonical relation $C^t \circ C$ of $F^*F$:  $C^t \circ C= \{ (x, \xi; y, \eta) | (x, \xi; s, t, \sigma, \tau) \in C^t; (s,t, \sigma, \tau; y , \eta) \in C  \}$. We have that $(s, t, \sigma, \tau; y, \eta) \in C$ implies
\begin{align}\label{Canonical relation of F}
\notag&t=\sqrt{(y_1-s)^2+y_2^2+h^2}+ \sqrt{(y_1+s)^2+y_2^2+h^2}\\
\notag&\sigma= \tau\Big{(}\frac{y_1-s}{\sqrt{(y_1-s)^2+y_2^2+h^2}}-\frac{y_1+s}{\sqrt{(y_1+s)^2+y_2^2+h^2}}\Big{)}\\
&\eta_1=\tau \Big{(}\frac{y_1-s}{\sqrt{(y_1-s)^2+y_2^2+h^2}}+\frac{y_1+s}{\sqrt{(y_1+s)^2+y_2^2+h^2}}\Big{)}\\
\notag&\eta_2=\tau \Big{(}\frac{y_2}{\sqrt{(y_1-s)^2+y_2^2+h^2}}+\frac{y_2}{\sqrt{(y_1+s)^2+y_2^2+h^2}}\Big{)}
\end{align}
and
$(x, \xi; s, t, \sigma, \tau) \in C^t$ implies
\begin{align}\label{Canonical relation of F*}
\notag&t=\sqrt{(x_1-s)^2+x_2^2+h^2}+ \sqrt{(x_1+s)^2+x_2^2+h^2}\\
\notag&\sigma= \tau(\frac{x_1-s}{\sqrt{(x_1-s)^2+x_2^2+h^2}}-\frac{x_1+s}{\sqrt{(x_1+s)^2+x_2^2+h^2}})\\
&\xi_1=\tau (\frac{x_1-s}{\sqrt{(x_1-s)^2+x_2^2+h^2}}+\frac{x_1+s}{\sqrt{(x_1+s)^2+x_2^2+h^2}})\\
\notag&\xi_2=\tau (\frac{x_2}{\sqrt{(x_1-s)^2+x_2^2+h^2}}+\frac{x_2}{\sqrt{(x_1+s)^2+x_2^2+h^2}})
\end{align}
From the first two relations in \eqref{Canonical relation of F} and \eqref{Canonical relation of F*}, we have
\Beq\label{Isorange curves}
\sqrt{(y_1-s)^2+y_2^2+h^2}+ \sqrt{(y_1+s)^2+y_2^2+h^2}=\sqrt{(x_1-s)^2+x_2^2+h^2}+ \sqrt{(x_1+s)^2+x_2^2+h^2}
\Eeq
and
\Beq\label{Isodoppler curves}
\frac{y_1-s}{\sqrt{(y_1-s)^2+y_2^2+h^2}}-\frac{y_1+s}{\sqrt{(y_1+s)^2+y_2^2+h^2}}=
\frac{x_1-s}{\sqrt{(x_1-s)^2+x_2^2+h^2}}-\frac{x_1+s}{\sqrt{(x_1+s)^2+x_2^2+h^2}}.
\Eeq We will use the prolate spheroidal coordinates to solve for $x$ and $y$. We let
\begin{align}\label{Prolate coordinate system}
\begin{array}{ll}
x_{1}=s\cosh \rho \cos \phi & y_{1}=s\cosh \rho' \cos \phi'\\
x_{2}=s\sinh \rho \sin \phi \cos \theta & y_{2}=s\sinh \rho' \sin \phi' \cos \theta'\\
x_{3}=h+ s\sinh \rho \sin \phi \sin \theta & y_{3}=h+ s\sinh \rho' \sin \phi' \sin \theta'
\end{array}
\end{align}
with $\rho > 0$, $0 \leq \phi \leq \pi$ and $0 \leq \theta < 2 \pi$.

In this case $x_3=0$ and we use it to solve for $h$. Hence $$(x_1-s)^2+x_2^2+h^2=s^2(\cosh \rho- \cos \phi)^2$$ and $$(x_1+s)^2+x_2^2+h^2=s^2(\cosh \rho+ \cos \phi)^2$$
Noting that $s>0$ and $\cosh\rho\pm\cos\phi>0$, the first relation given by \eqref{Isorange curves} in these coordinates becomes
$$s(\cosh \rho- \cos \phi)+s(\cosh \rho+ \cos \phi)=s(\cosh \rho'- \cos \phi')+s(\cosh \rho'+ \cos \phi') $$ from which we get $$\cosh \rho=\cosh \rho' \Rightarrow \rho=\rho'$$

The second relation given by \eqref{Isodoppler curves} becomes
$$\frac{\cosh \rho \cos \phi -1}{\cosh \rho-\cos \phi} - \frac{\cosh \rho \cos \phi +1}{\cosh \rho+\cos \phi}=\frac{\cosh \rho \cos \phi' -1}{\cosh \rho-\cos \phi'}- \frac{\cosh \rho \cos \phi' +1}{\cosh \rho+\cos \phi'} $$
After simplification we get
$$ \frac{\sin^2 \phi}{\cosh^2\rho-\cos^2 \phi} =\frac{\sin^2 \phi'}{\cosh^2\rho-\cos^2 \phi'}$$ which implies
$$(\cosh^2 \rho-1)(\sin^2 \phi-\sin^2\phi')=0 $$ Thus $\sin \phi= \pm \sin \phi' \Rightarrow \phi = \pm \phi', \pi \pm \phi'$.
\par We remark that $\cos \theta= \pm \sqrt{1-\frac{h^2}{s^2\sinh^2 \rho \sin^2 \phi}} =\pm \cos \theta'$ and note that $x_3=0$ implies that $\sin(\phi)\neq 0$, so that division by $\sin(\phi)$ is allowed here. We also remark that it is enough to consider  $\cos \theta=\cos \theta'$ as no additional relations are introduced by considering $\cos\theta=-\cos\theta'$.

Now we go back to $x$ and $y$ coordinates.  \par If $\phi'=\phi$ then
$x_1=y_1, \ x_2=y_2, \xi_i=\eta_i$ for $i=1,2$. For these points, the
composition, $C^t \circ C \subset \Delta=\{ (x, \xi; x, \xi) \} $.
\par If $\phi'=-\phi$ then $x_1=y_1, \ -x_2=y_2, \xi_1=\eta_1,
-\xi_2=\eta_2$. For these points, the composition, $C^t \circ C$ is a
subset of $ C_1=\{(x_1, x_2, \xi_1, \xi_2; x_1, -x_2, \xi_1, -\xi_2)
\} $ which is the graph of $\chi_1(x,\xi)=(x_1, -x_2, \xi_1, -\xi_2)$.
This in the base space represents the reflection about the $x_1$ axis.
\par If $\phi'=\pi-\phi$ then $-x_1=y_1, \ x_2=y_2, -\xi_1=\eta_1,
\xi_2=\eta_2$. For these points, the composition $C^t \circ C$ is a
subset of $ C_2=\{(x_1, x_2, \xi_1, \xi_2; -x_1, x_2, -\xi_1, \xi_2)
\} $ which is the graph of $\chi_2(x,\xi)=(-x_1, x_2, -\xi_1, \xi_2)$.
This in the base space represents the reflection about the $x_2$ axis.
\par If $\phi'=\pi+\phi$ then $-x_1=y_1, \ -x_2=y_2, -\xi_1=\eta_1,
-\xi_2=\eta_2$. For these points, the composition, $C^t \circ C$ is a
subset of $ C_3=\{(x_1, x_2, \xi_1, \xi_2; -x_1, -x_2, -\xi_1, -\xi_2)
\} $ which is the graph of $\chi_3(x,\xi)=(-x_1, -x_2, -\xi_1,
-\xi_2)$. This in the base space represents the reflection about the
origin.  \par Notice that $\chi_1 \circ \chi_1=\mbox{Id}, \ \chi_2
\circ \chi_2=\mbox{Id},\ \chi_1 \circ \chi_2=\chi_3$.  \par So far we
have obtained that $C^t \circ C \subset \Delta \cup C_1 \cup C_2 \cup
C_3$.

Next we consider the intersections of any two of these Lagrangians. We have
\par $\Delta$ intersects $C_1$ cleanly in codimension 2, $\Delta \cap C_1=\{(x, \xi; y, \eta)| x_{2}=0=\xi_{2} \} $ 
\par $\Delta$ intersects $C_2$ cleanly in codimension 2, $\Delta \cap C_2=\{(x, \xi; y, \eta)| x_1=0=\xi_1 \} $ 
\par $C_1$ intersects $C_3$ cleanly in codimension 2, $C_1 \cap C_3=\{(x, \xi; y, \eta)| x_1=0=\xi_1 \}$.
\par $C_2$ intersects $C_3$ cleanly in codimension 2, $C_2 \cap C_3=\{(x, \xi; y, \eta)| x_2=0=\xi_2 \}$.
\par $\Delta \cap C_3=\emptyset=C_1 \cap C_2$.
\epr

\begin{theorem}\label{theorem-decomposition} Let $F$ be as in \eqref{Modified Bi-static Mathematical Model} with order $m$. Then $F^*F$ can be decomposed as a sum of operators belonging in $I^{2m,0}(\Delta,C_1)+I^{2m,0}(\Delta, C_2)+ I^{2m,0}(C_{1},C_3)+I^{2m,0}(C_{2},C_{3})$.
\end{theorem}
\begin{proof}

Recall from Theorem \ref{FIOTheorem}, that the canonical relation of $F$ drops rank on the union of two sets, $\Sigma_{1}$ and $\Sigma_{2}$. Accordingly, we decompose $F$ into components such that the canonical relation of each component is either supported near a subset of the union of these two sets, one of these two sets or away from both these sets. More precisely, we let $\psi_{1}$ and $\psi_{2}$ be two infinitely differentiable functions defined as follows (refer Figure \ref{Cutoff Figure}):
\[
\psi_1(x)=\Bigg{\{} \begin{array}{cc} 1, & \mbox{ on }\  \{(x_{1},x_{2}): |x_2| < \epsilon  \} \\
0, & \mbox{ on } \ \{(x_{1},x_{2}): |x_2| > 2 \epsilon \}
\end{array}
\mbox{ and } \psi_2(x)=\Bigg{\{} \begin{array}{cc} 1, & \mbox{ on } \  \{(x_{1},x_{2}): |x_1| < \epsilon \} \\
0, & \mbox{ on } \ \{(x_{1},x_{2}): |x_1| > 2 \epsilon \}
\end{array}.
\]

\begin{figure}[htbp]
\centering
\includegraphics[height=10cm, width=10cm]{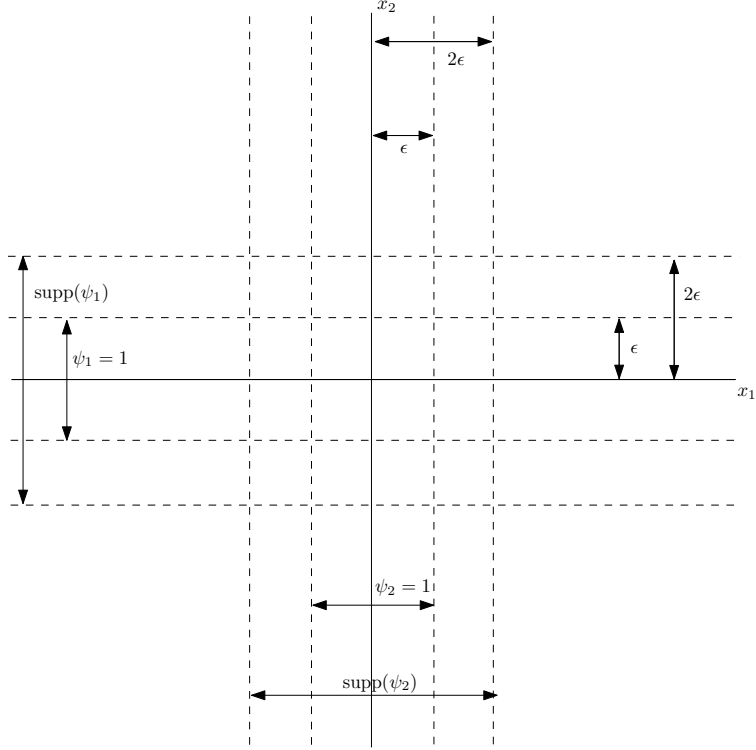}\caption{Support of
  the cutoff functions $\psi_{1}$ and $\psi_{2}$.}
\label{Cutoff Figure}
\end{figure}

Then we write $F=F_0+F_1+F_2+F_3$ where $F_i$  are given in terms of their kernels
\begin{align*}
&K_{F_0}=\int e^{-i\varphi} a \psi_1 \psi_2 \D \omega,\quad K_{F_1}=\int e^{-i\varphi} a \psi_1 (1-\psi_2) \D \omega, \\
&K_{F_2}=\int e^{-i\varphi} a(1-\psi_1) \psi_2 \D \omega,\quad K_{F_3}=\int e^{-i\varphi} a (1-\psi_1) (1-\psi_2) \D \omega,
\end{align*} where $\vp$ is the phase function of $F$ in \eqref{Modified Bi-static Mathematical Model}.
Now we consider $F^*F$, which using the decomposition of $F$ as above can be written as
\Beq\label{F*F Decomposition}
F^*F =F^*_0 F+(F_{1}+F_{2})^{*}F_{0}+ F^*_1F_1+ F^*_2F_2+F^*_1F_2+F^*_2F_1+F^*_1F_3+F^*_2F_3 +F^*_3F
\Eeq
The theorem now follows from Lemmas \ref{ImagingSection:Lemma F0}, \ref{ImagingSection:Lemma F1} and Theorem \ref{ImagingSection:I(p,l) regularity} below, where we analyze each of the compositions above.
\epr

\begin{lemma}\label{ImagingSection:Lemma F0}
$F_{0}$, $F_{1}^{*}F_{2}$ and $F_{2}^{*}F_{1}$ are smoothing operators.
\end{lemma}
\bpr We will only prove that $F_{0}$ and $F_{1}^{*}F_{2}$ are
smoothing. The proof for $F_{2}^{*}F_{1}$ is similar to that of
$F_{1}^{*}F_{2}$.  Let $ \wt{\vp}=\frac{1}{\omega}\vp$, where $\vp$ is the
phase function in \eqref{Modified Bi-static Mathematical Model}.

For $\delta=18\epsilon^{2}/h$, we analyze $F_{0}$ according to the following cases:

\begin{enumerate}[(a)]

\item\label{<} $\{(s,t):s>0, 0<t<2\sqrt{s^{2}+h^{2}}-\delta\}$.

For this case, we show that $K_{F_{0}}$ is smoothing. For, on
$\{(s,t): s>0, t<2\sqrt{s^{2}+h^{2}}-\delta\}$, $\wt{\vp}$ is bounded
away from $0$ and hence is a smooth function. Therefore for any $m\geq
0$
\[
\lb\frac{\I}{\wt{\vp}}\rb^{m}K_{F_{0}}(s,t,x)=\int \PD_{\omega}^{m}\lb e^{-\I \omega{\wt{\vp}}(s,t,x)}\rb
\psi_{1}(x)\psi_{2}(x)a(s,t,x,\omega) \D \omega.
\]
Now by integration by parts, the order of the amplitude can be made smaller than any negative number. Therefore $K_{F_{0}}$ is smoothing.
\item\label{=} $\{(s,t):s>0, |t-2\sqrt{s^{2}+h^{2}}|\leq\delta\}$.

For $(s,t)$ in this set, the kernel $K_{F_{0}}$ is identically $0$ due to our choice of the function $g(s,t)$ in \eqref{The function g(s,t)}.
\item\label{>} $\{(s,t):s>0, t> 2\sqrt{s^{2}+h^{2}}+\delta\}$.

In this case, we have that depending on our choice of $x$, the kernel
$K_{F_{0}}$ is either identically $0$ or smoothing. For, if we
consider $x$ in the complement of the set
$(-2\epsilon,2\epsilon)^{2}$, then due to the fact that
$\mbox{supp}(\psi_{1}(x)\psi_{2}(x))\subset[-2\epsilon,2\epsilon]^{2}$,
we have that the kernel is identically $0$. Now if we consider $x\in
(-2\epsilon,2\epsilon)^{2}$, then $\wt{\vp}$ is never vanishing. Then
by an integration by parts argument as in Case \eqref{<} above, we have that
$K_{F_{0}}$ is smoothing.
\end{enumerate}

Now we consider $F_{2}^{*}F_{1}$. We have
\[
K_{F_{1}}(s,t,x)=\int e^{-\I \omega \wt{\vp}(s,t,x)} \psi_{1}(x)(1-\psi_{2}(x)) a(s,t,x,\omega)\D \omega.
\]
and
\[
K_{F_{2}}^{*}(x,s,t)=\int e^{\I \omega\wt{\vp}(s,t,x)}(1-\psi_{1}(x))\psi_{2}(x)\overline{a(s,t,x,\omega)} \D \omega.
\]
Due to the cut-off functions $\psi_{1}$ and $\psi_{2}$ in these
kernels, we are only interested in those singularities lying above a
small neighborhood of the rectangles with vertices $(\pm
\epsilon,\pm\epsilon)$, $(\pm \epsilon,\pm 2\epsilon)$, $(\pm
2\epsilon,\pm \epsilon)$, $(\pm 2\epsilon, \pm 2\epsilon)$.

We have that $K_{F_1}$ is smoothing when $x$ values are restricted to
a small neighborhood of these rectangles. For, as in the previous
case, we consider the three cases: For Cases \eqref{<} and \eqref{>}, the kernel
$K_{F_1}$ is smoothing and the proof is identical as before. For Case
\eqref{=}, due to the choice of the function $g(s,t)$, the kernel
$K_{F_1}=0$. Therefore $F_{2}^{*}F_{1}$ is smoothing.

\epr
 \begin{lemma}\label{ImagingSection:Lemma F1}
  $F^*_1F_3$, $F^*_2 F_3$ and $F^*_3F$ can be decomposed as a sum of operators belonging to the space $I^{2m}(\Delta)+I^{2m}(C_{1}\setminus \Delta)+I^{2m}(C_{2}\setminus \Delta)+I^{2m}(C_{3}\setminus(C_{1}\cup C_{2}))$.
 \end{lemma}
 \bpr
 Each of these compositions is covered by the transverse intersection calculus. Below we will prove for the case of $F_{1}^{*}F_{3}$. For the other operators, the proofs are similar.

Let us decompose $F_{3}=F_{3}^{1}+F_{3}^{2}+F_{3}^{3}+F_{3}^{4}$ and $F_{1}^{*}=\lb F_{1}^{1}+F_{1}^{2}+F_{1}^{3}+F_{1}^{4}\rb ^{*}$, where the superscripts in both these sums denote restriction of $F_{3}$ and $F_{1}$, respectively, to each of the four quadrants. Note that in the decomposition of $F_{1}^{*}$, we stay away from $\Sigma_{1}$ by introducing a microlocal cutoff. This is valid because the support of the canonical relation of $F_{3}$ stays away from $\Sigma_{1}\cup \Sigma_{2}$. Then we have
\[
(F_{1}^{1})^{*}\NT F_{3}^{1}\in I^{2m}(\Delta), (F_{1}^{1})^{*}\NT F_{3}^{4}\in I^{2m}(C_{1}\setminus \Delta),  (F_{1}^{1})^{*}\NT F_{3}^{2} \in I^{2m}(C_{2}\setminus \Delta) \mbox{ and } (F_{1}^{1})^{*}\NT F_{3}^{3}\in I^{2m}(C_{3}\setminus(C_{1}\cup C_{2})).
\]
The other compositions can be considered similarly.
 \epr
We are left with the analysis of the compositions $F_{1}^{*}F_{1}$ and $F_{2}^{*}F_{2}$. This is the content of the next theorem:
\begin{theorem}
\label{ImagingSection:I(p,l) regularity}
 Let $F_{1}$ and $F_{2}$ be as above. Then
\begin{enumerate}[(a)]
\item $F_{1}^{*}F_{1}\in I^{2m,0}(\Delta, C_{1})+I^{2m,0}(C_{2},C_{3})$.
\item $F_{2}^{*}F_{2}\in I^{2m,0}(\Delta, C_{2})+I^{2m,0}(C_{1},C_{3})$.
\end{enumerate}
\end{theorem}
\bpr
We consider $F_{1}^{*}F_{1}$. The proof for $F_{2}^{*}F_{2}$ is similar.

We decompose $F_{1}$ by introducing a smooth cut-off function $\psi_{3}(x)$ such that $\psi_{3}(x)=1$ for $x_{1}>\epsilon/2$ and supported on the right-half plane $x_{1}\geq \epsilon/4$. That is,
we write
$F_{1}$ as
\[
F_{1}=F_{1}^{+}+F_{1}^{-},
\]
where
\begin{align*}
F_{1}^{+}V(s,t)=\int e^{-\I \vp(s,t,x,\omega)}\psi_{1}(x)(1-\psi_{2}(x))\psi_{3}(x)a(s,t,x,\omega)V(x)\D x
\end{align*}
and
\begin{align*}
F_{1}^{-}V(s,t)=\int e^{-\I \vp(s,t,x,\omega)}\psi_{1}(x)(1-\psi_{2}(x))(1-\psi_{3}(x))a(s,t,x,\omega)V(x)\D x
\end{align*}
Now
\Beq\label{F1*F1 decomposition}
F_{1}^{*}F_{1}=(F_{1}^{+})^{*}F_{1}^{+}+(F_{1}^{-})^{*}F_{1}^{+}+(F_{1}^{+})^{*}F_{1}^{-}+(F_{1}^{-})^{*}F_{1}^{-}.
\Eeq

The canonical relation of $F_{1}^{+}$ is a subset of \eqref{Common-midpoint Canonical Relation} with the additional condition that $x_{1}>\epsilon$. Then we have that
$WF((F_{1}^{+})^{*}F_{1}^{+})'\subset \Delta \cup C_{1}$. For, we already saw in Proposition \ref{Wavefront of composition} that $WF(F^{*}F)\subset \Delta \cup C_{1}\cup C_{2}\cup C_{3}$. In our case, imposing the additional restriction that $x_{1}>\epsilon$, the only contributions are in $\Delta$ and $C_{1}$. By a similar argument, we have that $WF((F_{1}^{-})^{*}F_{1}^{-})'\subset \Delta \cup C_{1}$.

Now let us consider the compositions $(F_{1}^{-})^{*}F_{1}^{+}$ and $(F_{1}^{+})^{*}F_{1}^{-}$. The wavefront sets of these operators are of the form $(x,\xi,y,\eta)$ such that $x_{1}$ and $y_{1}$ have opposite signs. We have already established in Proposition \ref{Wavefront of composition} that $|x_{i}|=|y_{i}|$ and $|\xi_{i}|=|\eta_{i}|$ for $i=1,2$. Now with the additional restriction that $x_{1}$ and $y_{1}$ have opposite signs (and therefore $\xi_{1}$ and $\eta_{1}$ have different signs as well), we have contributions contained in only $C_{2}$ and $C_{3}$.

The Lagrangian pairs $\Delta, C_{1}$ and $C_{2}, C_{3}$ intersect cleanly. Therefore there is a well-defined $I^{p,l}$ class -- which we will identify shortly -- in which each of the summands in \eqref{F1*F1 decomposition} lie.

We now show that
$(F_{1}^{+})^{*}F_{1}^{+}, (F_{1}^{-})^{*}F_{1}^{-}\in I^{2m,0}(\Delta,C_{1})$
and  that $ (F_{1}^{-})^{*}F_{1}^{+}, (F_{1}^{+})^{*}F_{1}^{-}\in I^{2m,0}(C_{2},C_{3})$. We follow the ideas of \cite{RF1}, where the iterated regularity theorem of  was used to prove an analogous result. The ideas of \cite{RF1} were recently employed to prove a similar result for a common-offset geometry in \cite{Krishnan-Quinto}. The proof we give is similar to the one given in \cite{Krishnan-Quinto}, but the phase function we work with is different.

We first consider the generator of the ideal of functions that vanish on $\Delta \cup C_{1}$ \cite{RF1}.
\begin{align*}
&\wt{p}_{1}=x_{1}-y_{1},\quad \wt{p}_{2}=x_{2}^{2}-y_{2}^{2},\quad \wt{p}_{3}=\xi_{1}-\eta_{1},\quad \wt{p}_{4}=(x_{2}+y_{2})(\xi_{2}-\eta_{2}),\\
&\wt{p}_{5}=(x_{2}-y_{2})(\xi_{2}+\eta_{2}),\quad \wt{p}_{6}=\xi_{2}^{2}-\eta_{2}^{2}.
\end{align*}

Let $p_{i}=q_{i}\wt{p}_{i}$, for $1\leq i\leq 6$, where $q_{1}, q_{2}$ are homogeneous of degree $1$ in $(\xi,\eta)$, $q_{3},q_{4}$ and $q_{5}$ are homogeneous of degree $0$ in $(\xi,\eta)$  and $q_{6}$ is homogeneous of degree $-1$ in $(\xi,\eta)$. Let $P_{i}$ be pseudodifferential operators with principal symbols $p_{i}$ for $1\leq i\leq 6$.

We show in Appendix \ref{Appendix:RegularityFormulas} that each $\wt{p}_{i}$ can be expressed in the following forms:
\begin{align}
\label{Identity 1}&\wt{p}_{1}=\frac{f_{11}(x,y,s)}{\omega}\PD_{s}\Phi+f_{12}(x,y,s)\PD_{\omega}\Phi\\
\label{Identity 2}&\wt{p}_{2}=\frac{f_{21}(x,y,s)}{\omega}\PD_{s}\Phi+f_{22}(x,y,s)\PD_{\omega}\Phi\\
\label{Identity 3}&\wt{p}_{3}= f_{31}(x,y,s)\PD_{s}\Phi+\omega f_{32}(x,y,s)\PD_{\omega}\Phi\\
\label{Identity 4}&\wt{p}_{4}=f_{41}(x,y,s)\PD_{s}\Phi+\omega f_{42}(x,y,s)\PD_{\omega}\Phi\\
\label{Identity 5}&\wt{p}_{5}=f_{51}(x,y,s)\PD_{s}\Phi+\omega f_{52}(x,y,s)\PD_{\omega}\Phi\\
\label{Identity 6}&\wt{p}_{6}=\omega f_{61}(x,y,s)\PD_{s}\Phi+\omega^{2}f_{62}(x,y,s)\PD_{\omega}\Phi
\end{align}
where $f_{ij}$ for $1\leq i\leq 6$ and $j=1,2$ are  smooth functions.

Now the rest of the proof is the same as in \cite[Theorem 1.6]{RF1}. We give it for completeness.

Let $K^{+}_{1}$ be the kernel of $(F_{1}^{+})^{*}F_{1}^{+}$. This is the kernel in \eqref{Kernel of composed operator}, but with $\wt{a}$ there replaced by $\psi_{1}(x)(1-\psi_{2}(x))\psi_{3}(x)\psi_{1}(y)(1-\psi_{2}(y))\psi_{3}(y)\wt{a}(x,y,s,\omega)$. For simplicity, we rename this as $\wt{a}$ again.

We then have that $\tilde{a} \in S^{2m+1}$ and
\begin{align*}
P_{1}K_{1}^{+}(x,y)&=\int e^{\I \Phi(x,y,s,\omega)}\wt{a}(x,y,s,\omega)q_{1}\lb \frac{f_{11}(x,y,s)}{\omega}\PD_{s}\Phi+f_{12}(x,y,s)\PD_{\omega}\Phi\rb \D s\D \omega\\
&=\int \PD_{s}\lb e^{\I \Phi(x,y,s,\omega)}\rb\frac{q_{1}}{\I \omega}\wt{a}(x,y,s,\omega)f_{11}(x,y,s)\D s\D \omega\\
&+\int \PD_{\omega}\lb e^{\I \Phi(x,y,s,\omega)}\rb\frac{q_{1}}{\I}\wt{a}(x,y,s,\omega)f_{12}(x,y,s)\D s\D \omega\\
\intertext{By integration by parts}
&=-\Bigg{\{}\int e^{\I \Phi(x,y,s,\omega)}\PD_{s}\lb \frac{q_{1}}{\I \omega}\wt{a}(x,y,s,\omega)f_{11}(x,y,s)\rb \D s\D \omega\\
&+\int e^{\I \Phi(x,y,s,\omega)}\PD_{\omega}\lb\frac{q_{1}}{\I}\wt{a}(x,y,s,\omega)f_{12}(x,y,s)\rb \D s\D \omega \Bigg{\}}.
\end{align*}
Note that $q_{1}$ is homogeneous of degree $1$ in $\omega$, and $\wt{a}$ is a symbol of order $2m+1$, hence each amplitude term in the sum above is of order $2m+1$.

Therefore by Definition \ref{Ipl3}, we have that $P_{1}K_{1}^{+} \in H^{s_{0}}_{\mathrm{loc}}$ for some $s_{0}$.

A similar argument works for each of the other five pseudodifferential
operators. Hence by Proposition \ref{Ipl3}, we have that
$(F_{1}^{+})^{*}F_{1}^{+}\in I^{p,l}(\Delta,C_1)$. Because $C$ is a
local canonical graph away from $\Sigma$, the transverse intersection
calculus applies for the composition $(F_{1}^{+})^{*}F_{1}^{+}$ away
from $\Sigma$. Hence $(F_{1}^{+})^{*}F_{1}^{+}$ is of order $2m$ on
$\Delta\setminus C_1$ and $C_1 \setminus \Sigma$. Since
$(F_{1}^{+})^{*}F_{1}^{+}$ is of order $p+l$ on $\Delta\setminus
\Sigma$ and is of order $p$ on $C_1 \setminus \Sigma$, we have that
$p=2m$ and $l=0$. Therefore $(F_{1}^{+})^{*}F_{1}^{+}\in
I^{2m,0}(\Delta,C_{1})$. Similarly $(F_{1}^{-})^{*}F_{1}^{-} \in
I^{2m,0}(\Delta,C_{1})$.

To show that  $(F_{1}^{-})^{*}F_{1}^{+}, (F_{1}^{+})^{*}F_{1}^{-}\in I^{2m,0}(C_{2},C_{3})$ we can use the iterated regularity result as above.

The generators of the ideal of functions that vanish on $C_2 \cup C_3$ are:
\[ \wt{r}_{1}=x_{1}+y_{1},  \wt{r}_{2}=\xi_{1}+\eta_{1} \mbox{ and }   \tilde{p}_2, \  \tilde{p}_4,  \ \tilde{p}_5,\  \tilde{p}_6  \mbox{ are the same as in } \eqref{Identity 2}, \eqref{Identity 4}, \eqref{Identity 5} \mbox{ and } \eqref{Identity 6} \mbox{ respectively}.
\]
Four of the functions in the ideal are the same as in the proof above and we can find similar expressions for the first two.

However we will also give an alternate proof below.

\epr

\begin{proposition}\label{I(p,l) regularity of cross terms}
$(F_{1}^{-})^{*}F_{1}^{+}, (F_{1}^{+})^{*}F_{1}^{-}\in I^{2m,0}(C_{2},C_{3})$.
\end{proposition}
\bpr
We show for $(F_{1}^{-})^{*}F_{1}^{+}$. The proof for the other case is similar.
Consider the operator $R$ defined as follows:
\[
RV(x_{1},x_{2})=V(-x_{1},x_{2}).
\]
This is a Fourier integral operator of order $0$ with the canonical relation $C_{2}$. This is because,
\begin{align*}
RV(x_{1},x_{2})&=\int e^{\I (x-y)\cdot \xi} R_{2}V(y_{1},y_{2}) \D y \D \xi\\
&=\int e^{\I (x-y)\cdot \xi} V(-y_{1},y_{2}) \D y \D \xi\\
&=\int e^{\I( (x_{1}+y_{1})\xi_{1}+(x_{2}-y_{2})\xi_{2})} V(y_{1},y_{2}) \D y \D \xi
\end{align*}
It is easy to check that canonical relation is $C_{2}$.

Now consider the operator $\wt{F}=F_{1}^{-}\circ R$. This is given by
\begin{align*}
\wt{F}V(s,t)&=\int e^{-\I \vp(s,t,x,\omega)}(\psi_{1}(1-\psi_{2})(1-\psi_{3}))(-x_{1},x_{2}))a(s,t,-x_{1},x_{2},\omega)V(x)\D x\D \omega\\
\intertext{
Note that $(1-\psi_{3})(-x_{1},x_{2})=\psi_{3}(x_{1},x_{2})$ except in a small neighborhood of the origin. Since $\psi_{1}(1-\psi_{2})$ is $0$ in a neighborhood of the origin, and noting that we can arrange $\psi_{1}$ and $\psi_{2}$ to be symmetric with respect to $x_{1}$, we have}
&=\int e^{-\I \vp(s,t,x,\omega)}(\psi_{1}(1-\psi_{2})\psi_{3}))(x)a(s,t,-x_{1},x_{2},\omega)V(x)\D x\D \omega.
\end{align*}
Now we have that $\wt{F}^{*}F_{1}^{+}\in I^{2m,0}(\Delta,C_{1})$. In fact the kernel of this operator has the same form as in \eqref{Kernel of composed operator} and the same proof as in Theorem \ref{ImagingSection:I(p,l) regularity} applies.
Next we use \cite[Proposition 4.1]{Guillemin-Uhlmann} to show that $R^{*}\wt{F}^{*}F_{1}^{+}\in I^{2m,0}(C_{2},C_{3})$. It is straightforward to check that $C_{2}\circ \Lambda=C_{2}$, $C_{2}\circ C_{1}=C_{3}$ and $C_{2}\times \Delta$ (as well as $C_{2}\times C_{1}$) intersects $T^{*}X\times \Delta_{T^{*}{X}}\times T^{*}X$ transversally. Hence the hypotheses of \cite[Proposition 4.1]{Guillemin-Uhlmann} are verified and we conclude that $R^{*}\wt{F}^{*}F_{1}^{+}\in I^{2m,0}(C_{2},C_{3})$. Since $\wt{F}^{*}=R^{*}(F_{1}^{-})^{*}$ and $(R^{*})^{2}=\mbox{Id}$ we have $(F_{1}^{-})^{*}F_{1}^{+}\in I^{2m,0}(C_2, C_3)$.
\epr
Since $I^{2m}(\Delta) \in I^{2m,0}(\Delta, C_1);  I^{2m}(C_{1} \setminus \Delta) \in  I^{2m,0}(\Delta, C_1);  I^{2m}(C_{2} \setminus \Delta) \in I^{2m,0}(\Delta, C_2)$ and  $I^{2m}(C_{3} \setminus(C_{1} \cup C_{2})) \in I^{2m,0}(C_1, C_3)$  then
Theorem \ref{theorem-decomposition} follows using
Lemmas \ref{ImagingSection:Lemma F0}, \ref{ImagingSection:Lemma F1},
Theorem 5.5 and Proposition \ref{I(p,l) regularity of cross terms}.\\

 \begin{remark}\label{remark:strength} Using the properties of the
$I^{p,l}$ classes, $F^*F \in I^{2m,0} (\Delta, C_1)$ implies that
$F^*F \in I^{2m} (\Delta \setminus C_1)$ and $F^*F \in I^{2m} (C_1
\setminus \Delta)$. This means that $F^*F$ has the same order on both
$\Delta$ and $C_1$ which implies that the artifact $C_1$ has the same
strength as the initial singularities given by $\Delta$. Similarly for
$C_2$ and $C_3$. Note that $C_1$ gives an artifact that is a
reflection in the $x_1$ axis, $C_2$ gives an artifact that is a
reflection in the $x_2$ axis, and $C_3$ gives an artifact that is a
reflection in the origin.
\end{remark}

\section{Acknowledgements}

All authors thank The American Institute of Mathematics (AIM) for the
SQuaREs (Structured Quartet Research Ensembles) award, which enabled
their research collaboration, and for the hospitality during the
authors' visit to AIM in April 2011.  Most of the results in this
paper were obtained during that visit.  Both Ambartsoumian and Quinto
thank the Mathematical Sciences Research Institute (MSRI) for their
hospitality during the Inverse Problems Special Semester in 2010
during which this material was discussed.

Ambartsoumian was supported in part by NSF Grant DMS 1109417,
NHARP Consortium Grant 003656-0109-2009, and DOD CDMRP Grant
BC-063989.

Felea was supported in part by Simons Foundation grant 209850.

Krishnan was supported in part by NSF grants DMS 1129154 and DMS
1109417, the first being a post-doctoral supplement to Quinto's NSF
grant DMS 0908015. Additionally he thanks Tufts University for
providing an excellent research environment and the University of
Bridgeport for the support he received as a faculty member there.

Nolan was supported in part by The Mathematics Applications
Consortium for Science and Industry (MACSI) which is funded by the Science
Foundation Ireland Mathematics Initiative Grant 06/MI/005.

Quinto was partially supported by NSF grant DMS 0908015.

\appendix

\section{}\label{Appendix:RegularityFormulas}
Here we give the derivations of \eqref{Identity 1} - \eqref{Identity
6}. We will work in the coordinate system defined in \eqref{Prolate
coordinate system}. We extend the phase function in \eqref{Phase of
kernel} to $\Rb^{3}$ by letting
\begin{align*}
\wt{\Phi}=\omega\Bigg{\{}&\sqrt{(y_{1}-s)^{2}+y_{2}^{2}+(y_{3}-h)^{2}}+\sqrt{(y_{1}+s)^{2}+y_{2}^{2}+(y_{3}-h)^{2}}\:\:-\\
&\Big{(}\sqrt{(x_{1}-s)^{2}+x_{2}^{2}+(x_{3}-h)^{2}}+\sqrt{(x_{1}+s)^{2}+x_{2}^{2}+(x_{3}-h)^{2}}\Big{)}\Bigg{\}}.
\end{align*}
Then note that
\[
\PD_{\omega}\wt{\Phi}|_{x_{3}=y_{3}=0}=\PD_{\omega}\Phi\quad \mbox{and} \quad \PD_{s}\wt{\Phi}|_{x_{3}=y_{3}=0}=\PD_{s}\Phi.
\]
\subsection{Expression for $x_{1}-y_{1}$} We obtain an expression for $x_{1}-y_{1}$ in the form
\[
A_{1}:=x_{1}-y_{1}=\frac{f_{11}(x,y,s)}{\omega}\PD_{s}\Phi+f_{12}(x,y,s)\PD_{\omega}\Phi,
\]
where $f_{11}$ and $f_{12}$ are smooth functions. In the coordinate system \eqref{Prolate coordinate system},
\[
A_{1}=s(\cosh \rho \cos \phi -\cosh \rho'\cos\phi')
\]
\[
\PD_{\omega}\wt{\phi}=2s(\cosh \rho'-\cosh \rho).
\]
\begin{align*}
\PD_{s}\wt{\Phi}&=\omega\Bigg{\{} \Big{(}\frac{\cosh \rho'\cos\phi'+1}{\cosh \rho'+\cos\phi'}-\frac{\cosh \rho'\cos\phi'-1}{\cosh \rho'-\cos\phi'}\Big{)}-\Big{(}\frac{\cosh \rho\cos\phi+1}{\cosh \rho+\cos\phi}-\frac{\cosh \rho\cos\phi-1}{\cosh \rho-\cos\phi}\Big{)}\Bigg{\}}\\
&=2\omega\Bigg{\{} \frac{\cosh\rho'-\cosh\rho'\cos^{2}\phi'}{\cosh^{2}\rho'-\cos^{2}\phi'}-\frac{\cosh\rho-\cosh\rho\cos^{2}\phi}{\cosh^{2}\rho-\cos^{2}\phi}\Bigg{\}}
\intertext{After simplifying, we get,}
&=2\omega\Bigg{\{}\frac{(\cosh \rho-\cosh \rho')(\cosh\rho\cosh\rho'-\cos^{2}\phi\cos^{2}\phi')} {(\cosh^{2}\rho'-\cos^{2}\phi')(\cosh^{2}\rho-\cos^{2}\phi)} +\\
&\hspace{0.15in}\frac{(\cosh\rho'\cos^{2}\phi-\cosh\rho\cos^{2}\phi')(\cosh\rho\cosh\rho'-1)}{(\cosh^{2}\rho'-\cos^{2}\phi')(\cosh^{2}\rho-\cos^{2}\phi)}\Bigg{\}}.
\end{align*}
Now observing that $\cosh\rho'-\cosh \rho=\frac{\PD_{\omega}\wt{\Phi}}{2s}$ and adding and subtracting $\cosh\rho\cos^{2}\phi$ to the second term on the right above, we have,
\begin{align*}
&\PD_{s}\wt{\Phi}+\frac{\omega}{s}\frac{(\cosh\rho\cosh\rho'-\cos^{2}\phi\cos^{2}\phi')}{(\cosh^{2}\rho'-\cos^{2}\phi')(\cosh^{2}\rho-\cos^{2}\phi)}\PD_{\omega}\wt{\Phi}=\\
&2\omega\frac{\Big{\{}(\cosh\rho'-\cosh\rho)\cos^{2}\phi+\cosh\rho(\cos^{2}\phi-\cos^{2}\phi')\Big{\}}(\cosh\rho\cosh\rho'-1)}{(\cosh^{2}\rho'-\cos^{2}\phi')(\cosh^{2}\rho-\cos^{2}\phi)}.
\end{align*}
From this we get
\begin{align}\label{Diff. of cosines}
\cos\phi-\cos\phi'=&\frac{\PD_{s}\wt{\Phi}(\cosh^{2}\rho'-\cos^{2}\phi')(\cosh^{2}\rho-\cos^{2}\phi)}{2\omega\cosh\rho(\cosh\rho'\cosh\rho-1)(\cos\phi+\cos\phi')}\\
\notag&+\frac{\frac{\omega}{s}\PD_{\omega}\wt{\Phi}\left((\cosh\rho\cosh\rho'-\cos^{2}\phi\cos^{2}\phi')-\cos^{2}\phi(\cosh\rho\cosh\rho'-1)\right)}{2\omega\cosh\rho(\cosh\rho'\cosh\rho-1)(\cos\phi+\cos\phi')}.
\end{align}
Now note that
\begin{align*}
A_{1}=&\frac{s\PD_{s}\wt{\Phi}(\cosh^{2}\rho'-\cos^{2}\phi')(\cosh^{2}\rho-\cos^{2}\phi)}{2\omega(\cosh\rho'\cosh\rho-1)(\cos\phi+\cos\phi')}\\
&+\frac{\omega\PD_{\omega}\wt{\Phi}\left((\cosh\rho\cosh\rho'-\cos^{2}\phi\cos^{2}\phi')-\cos^{2}\phi(\cosh\rho\cosh\rho'-1)\right)}{2\omega(\cosh\rho'\cosh\rho-1)(\cos\phi+\cos\phi')}-\frac{\cos\phi'}{2}\PD_{\omega}\wt{\Phi}\\
\intertext{Now letting $x_{3}=y_{3}=0$, we see that we have written $x_{1}-y_{1}$ as a combination in terms of $\PD_{\omega}\Phi$ and $\PD_{s}\Phi$ as follows:}
&=\frac{s(\cosh^{2}\rho'-\cos^{2}\phi')(\cosh^{2}\rho-\cos^{2}\phi)}{2(\cosh\rho'\cosh\rho-1)(\cos\phi+\cos\phi')}\frac{\PD_{s}\Phi}{\omega}+\\
&\Big{\{}\frac{\left((\cosh\rho\cosh\rho'-\cos^{2}\phi\cos^{2}\phi')-\cos^{2}\phi(\cosh\rho\cosh\rho'-1)\right)}{2(\cosh\rho'\cosh\rho-1)(\cos\phi+\cos\phi')}-\frac{\cos\phi'}{2}\Big{\}}\PD_{\omega}\Phi
\end{align*}
We can write the above expression in the Cartesian coordinate system. First, for simplicity, let
\begin{align*}
X_{1}=\sqrt{(x_{1}-s)^{2}+x_{2}^{2}+h^{2}}\\
X_{2}=\sqrt{(x_{1}+s)^{2}+x_{2}^{2}+h^{2}}
\end{align*}
with $Y_{1}$ and $Y_{2}$ being similarly defined with $x$ replaced by $y$. Then we have

\begin{align*}
x_{1}-&y_{1}=\frac{s \lb\frac{Y_{1}Y_{2}}{s^{2}}\rb\lb\frac{X_{1}X_{2}}{s^{2}}\rb}{2\Big{(} \lb\frac{Y_{1}+Y_{2}}{2s}\rb\lb\frac{X_{1}+X_{2}}{2s}\rb-1\Big{)}\Big{(} \frac{x_{1}}{\frac{X_{1}+X_{2}}{2}}+\frac{y_{1}}{\frac{Y_{1}+Y_{2}}{2}}\Big{)}}\frac{\PD_{s}\Phi}{\omega}+ \\
&\left\{\frac{\Big{(}\lb \frac{X_{1}+X_{2}}{2s}\rb\lb
\frac{Y_{1}+Y_{2}}{2s}\rb-\frac{x_{1}^{2}y_{1}^{2}}{\lb
\frac{X_{1}+X_{2}}{2}\rb^{2}\lb
\frac{X_{1}+X_{2}}{2}\rb^{2}}\Big{)}-\Big{(}
\frac{x_{1}^{2}}{\lb\frac{X_{1}+X_{2}}{2}\rb^{2}}\Big{(}
\lb\frac{Y_{1}+Y_{2}}{2s}\rb\lb\frac{X_{1}+X_{2}}{2s}\rb-1\Big{)}\Big{)}}{2\Big{(} \lb\frac{Y_{1}+Y_{2}}{2s}\rb\lb\frac{X_{1}+X_{2}}{2s}\rb-1\Big{)}\Big{(} \frac{x_{1}}{\frac{X_{1}+X_{2}}{2}}+\frac{y_{1}}{\frac{Y_{1}+Y_{2}}{2}}\Big{)}}-\frac{y_{1}}{Y_{1}+Y_{2}}\right\}\PD_{\omega}\Phi.
\end{align*}

\subsection{Expression for $x_{2}^{2}-y_{2}^{2}$} Now we write $x_{2}^{2}-y_{2}^{2}$ in the form
\begin{equation}\label{A2}
A_{2}:=x_{2}^{2}-y_{2}^{2}=\frac{f_{21}(x,y,s)}{\omega}\PD_{s}\Phi+f_{22}(x,y,s)\PD_{\omega}\Phi,
\end{equation}
where $f_{21}$ and $f_{22}$ are smooth functions.
  $A_{2}$ in the coordinate system \eqref{Prolate coordinate system} is
\begin{align}
\notag A_{2}&=s^{2}\lb \sinh^{2}\rho\sin^{2}\phi\cos^{2}\theta-\sinh^{2}\rho'\sin^{2}\phi'\cos^{2}\theta'\rb\\
\label{First x2-y2 term}&=s^{2}\lb \sinh^{2}\rho\sin^{2}\phi-\sinh^{2}\rho'\sin^{2}\phi'\rb+\\
\label{Second x2-y2 term}&s^{2}\lb \sinh^{2}\rho'\sin^{2}\phi'\sin^{2}\theta'-\sinh^{2}\rho\sin^{2}\phi\sin^{2}\theta\rb.
\end{align}
For $x_{3}=y_{3}=0$, \eqref{Second x2-y2 term} is $0$. Therefore we focus only on the term \eqref{First x2-y2 term}, which we still denote as $A_{2}$, and obtain an expression of the form \eqref{A2} for this term.

Using the formulas $\sinh^{2}\rho=\cosh^{2}\rho -1$, $\sin^{2}\phi=1-\cos^{2}\phi$, and simplifying, we have
\begin{align*}
A_{2}&=s^{2}\lb(\cosh^{2}\rho-\cosh^{2}\rho')\sin^{2}\phi-(\cos^{2}\phi-\cos^{2}\phi')\sinh^{2}\rho'\rb\\
&=s^{2}\lb (\cosh \rho-\cosh\rho')(\cosh \rho+\cosh\rho')-(\cos \phi-\cos\phi')(\cos\phi+\cos\phi')\sinh^{2}\rho'\rb.
\end{align*}
Recall that $\cosh\rho-\cos\rho' = \frac{\PD_{\omega} \wt{\phi}}{2s}$ and using the  expression for $\cos\phi-\cos\phi'$ in \eqref{Diff. of cosines}, and setting $x_{3}=y_{3}=0$, we see that $x_{2}^{2}-y_{2}^{2}$ can be written in the form \eqref{A2}.
\subsection{Expression for $\xi_{1}-\eta_{1}$} Note that $\xi_{1}=\PD_{x_{1}}\Phi$ and $\eta_{1}=-\PD_{y_{1}}\Phi$ and so $A_{3}:=\xi_{1}-\eta_{1}=\PD_{x_{1}}\Phi+\PD_{y_{1}}\Phi$.

We have
\begin{align*}
A_{3}&=\omega\Bigg{\{} \Big{(} \frac{y_{1}-s}{|y-\g_{T}(s)|}+\frac{y_{1}+s}{|y-\g_{R}(s)|}\Big{)}-\Big{(} \frac{x_{1}-s}{|x-\g_{T}(s)|}+\frac{x_{1}+s}{|x-\g_{R}(s)|}\Big{)}\Bigg{\}}.
\intertext{In the coordinate system \eqref{Prolate coordinate system} this is}
&=2\omega\Bigg{(} \frac{\sinh^{2}\rho'\cos\phi'}{\cosh^{2}\rho'-\cos^{2}\phi'}-\frac{\sinh^{2}\rho\cos\phi}{\cosh^{2}\rho-\cos^{2}\phi}\Bigg{)}
\intertext{Simplifying this, we get}
&=2\omega \Bigg{\{} \frac{(\cos \phi-\cos\phi')(-\sinh^{2}\rho\cosh^{2}\rho'-\sinh^{2}\rho'\cos\phi\cos\phi')}{(\cosh^{2}\rho-\cos^{2}\phi)(\cosh^{2}\rho'-\cos^{2}\phi')}\\
&+\frac{(\cosh \rho-\cosh \rho')(\cosh \rho+\cosh \rho')\cos \phi'(1-\cos \phi\cos\phi')}{(\cosh^{2}\rho-\cos^{2}\phi)(\cosh^{2}\rho'-\cos^{2}\phi')}\Bigg{\}}.
\end{align*}
Now noting that $\cosh \rho-\cosh\rho'=\frac{\PD_{\omega}\wt{\Phi}}{2s}$ and using the formula \eqref{Diff. of cosines} for $\cos\phi-\cos\phi'$ and setting $x_{3}=y_{3}=0$, we can write $A_{3}$ in the form \eqref{Identity 3}.
\subsection{Expression for $(x_{2}-y_{2})(\xi_{2}+\eta_{2})$} Using the coordinate system \eqref{Prolate coordinate system}, we can write $A_{4}:= (x_{2}+y_{2})(\xi_{2}-\eta_{2})$ (up to a negative sign) as
\begin{align*}
(x_{2}-y_{2})(\xi_{2}+\eta_{2})&=\omega(x_{2}-y_{2})\lb
\frac{x_{2}}{\norm{x-\g_{T}}}+\frac{x_{2}}{|x-\g_{R}|}+\frac{y_{2}}{\norm{y-\g_{T}}}+\frac{y_{2}}{\norm{y-\g_{R}}}\rb\\
&=
\frac{2\omega}{s}\Bigg{(} \frac{x_{2}^{2}\cosh \rho}{\cosh^{2}\rho-\cos^{2}\theta}-\frac{y^{2}\cosh \rho'}{\cosh^{2}\rho'-\cos^{2}\theta'}\\
&\hspace{0.3in}+\frac{x_{2}y_{2}\cosh\rho'}{\cosh^{2}\rho'-\cos^{2}\theta'}-\frac{x_{2}y_{2}\cosh \rho}{\cosh^{2}\rho-\cos^{2}\theta}\Bigg{)}\\
&=\frac{2\omega}{s}\Bigg{(} \frac{x_{2}^{2}\cosh \rho}{\cosh^{2}\rho-\cos^{2}\theta}-\frac{x_{2}^{2}\cosh \rho'}{\cosh^{2}\rho'-\cos^{2}\theta'}\\
&+(x_{2}^{2}-y_{2}^{2})\frac{\cosh \rho'}{\cosh^{2}\rho'-\cos^{2}\theta'}\\
&+
\frac{x_{2}y_{2}\cosh\rho'}{\cosh^{2}\rho'-\cos^{2}\theta'}-\frac{x_{2}y_{2}\cosh \rho}{\cosh^{2}\rho-\cos^{2}\theta}\Bigg{)},\\
\end{align*}
Here we have added and subtracted $\frac{x_{2}^{2}\cosh \rho'}{\cosh^{2}\rho'-\cos^{2}\theta'}$ in the previous equation. Simplifying this we get,
\begin{align*}
(x_{2}-y_{2})(\xi_{2}+\eta_{2})=&
\frac{2\omega}{s}\Bigg{(}(x_{2}^{2}-x_{2}y_{2})\Bigg[
\frac{(\cosh \rho\cosh \rho'+\cos^{2}\theta)(\cosh \rho'-\cosh \rho)}{(\cosh^{2}\rho-\cos^{2}\theta)(\cosh^{2}\rho'-\cos^{2}\theta')}\\
&\quad\quad\phantom{(x_{2}^{2}-x_{2}y_{2})}+\frac{\cosh \rho(\cos\theta+\cos\theta')(\cos
\theta-\cos\theta')}{(\cosh^{2}\rho-\cos^{2}\theta)(\cosh^{2}\rho'-\cos^{2}\theta')}\bigg]\\
&+(x_{2}^{2}-y_{2}^{2})\frac{\cosh \rho'}{\cosh^{2}\rho'-\cos^{2}\theta'}\Bigg{)}.
\end{align*}

Now note that  $\cosh \rho'-\cosh \rho=\frac{\PD_{\omega}\Phi}{2s}$ and we already have expressions for $\cos\theta-\cos\theta'$ (Equation \eqref{Diff. of cosines}) and for $x_{2}^{2}-y_{2}^{2}$ involving combinations of $\PD_{\omega}\Phi$ and $\PD_{s}\Phi$.

Hence we can write $(x_{2}-y_{2})(\xi_{2}+\eta_{2})$ in the form
 of \eqref{Identity 4}.
Note that our calculation in this section shows
that \bel{handy-id}\begin{aligned}\frac{\cosh
\rho}{\cosh^{2}\rho-\cos^{2}\theta}-\frac{\cosh
\rho'}{\cosh^{2}\rho'-\cos^{2}\theta'} &= \frac{(\cosh \rho\cosh
\rho'+\cos^{2}\theta)(\cosh \rho'-\cosh
\rho)}{(\cosh^{2}\rho-\cos^{2}\theta)(\cosh^{2}\rho'-\cos^{2}\theta')}\\
&\qquad+\frac{\cosh \rho(\cos\theta+\cos\theta')(\cos
\theta-\cos\theta')}
{(\cosh^{2}\rho-\cos^{2}\theta)(\cosh^{2}\rho'-\cos^{2}\theta')}\end{aligned}\ee
This will be useful in the derivation of  \eqref{Identity 6} in \ref{Appendix:Identity6} below.
\subsection{Expression for $(x_{2}+y_{2})(\xi_{2}-\eta_{2})$}
This is very similar to the derivation of the expression we obtained for $(x_{2}-y_{2})(\xi_{2}+\eta_{2})$.
\subsection{Expression for $\xi_{2}^{2}-\eta_{2}^{2}$}\label{Appendix:Identity6}
We have
\begin{align*}
\xi_{2}^{2}-\eta_{2}^{2}&=\omega^{2}\lb
\lb\frac{x_{2}}{|x-\g_{T}|}+\frac{x_{2}}{|x-\g_{R}|}\rb^2-\lb\frac{y_{2}}{|y-\g_{T}|}+\frac{y_{2}}{|y-\g_{R}|}\rb^2\rb\\
&={4}\omega^{2}\lb x_{2}^{2}\frac{\cosh^{2}\rho}{(\cosh^{2}\rho-\cos^{2}\theta)^{2}}-y_{2}^{2}\frac{\cosh^{2}\rho'}{(\cosh^{2}\rho'-\cos^{2}\theta')^{2}}\rb\\
&={4}\omega^{2}\Bigg{\{} x_{2}^{2}\lb
\frac{\cosh^{2}\rho}{(\cosh^{2}\rho-\cos^{2}\theta)^{2}}-\frac{\cosh^{2}\rho'}{(\cosh^{2}\rho'-\cos^{2}\theta')^{2}}\rb\\
&+(x_{2}^{2}-y_{2}^{2})\frac{\cosh^{2}\rho'}{(\cosh^{2}\rho'-\cos^{2}\theta')^{2}}\Bigg{\}}.
\end{align*}
Now using the computations for $x_{2}^{2}-y_{2}^{2}$ and
$(x_{2}-y_{2})(\xi_{2}+\eta_{2})$, in particular \eqref{handy-id},
 we can write $\xi_{2}^{2}-\eta_{2}^{2}$ in the form
\[
\xi_{2}^{2}-\eta_{2}^{2}=\omega f_{61}(x,y,s)\PD_{s}\Phi+\omega^{2}f_{62}(x,y,s)\PD_{\omega}\Phi
\]
for smooth functions $f_{61}, f_{62}$.\\

\section{}\label{Factor in g(s,t)}
Here we explain the reason for setting $g(s,t)=0$ for $|t-2\sqrt{s^{2}+h^{2}}|<20\epsilon^{2}/h$
in \eqref{The
function g(s,t)}.

In the proof of Theorem \ref{theorem-decomposition} -- more precisely Lemma \ref{ImagingSection:Lemma F0} --  recall that we consider four squares with vertices $(\pm \epsilon,\pm\epsilon)$, $(\pm \epsilon,\pm2\epsilon)$, $(\pm 2 \epsilon,\pm\epsilon)$,
$(\pm 2\epsilon,\pm 2\epsilon)$.  The motivation to choose $g=0$ as above comes from the fact that we want the amplitude term $a$ of $F$ to be $0$ for those $(s,t)$ such that the ellipse defined by it is contained in a small neighborhood containing these squares.

One way to find this is as follows:

Given $(s,t)$, the ellipse $\sqrt{(x_{1}-s)^{2}+x_{2}^{2}+h^{2}}+\sqrt{(x_{1}+s)^{2}+x_{2}^{2}+h^{2}}=t$ can be written in the form
\[
(4t^2-16s^2)x_{1}^{2}+4t^{2}x_{2}^{2}=t^{4}-4t^{2}(s^{2}+h^{2}).
\]
Note that for this ellipse, the length of the semi-minor axis is always smaller that the length of the semi-major axis. The point $(2\epsilon,2\epsilon)$ is $2\sqrt{2}\epsilon$ away from the origin. Therefore let us choose a $t$ for which the ellipse passes through the point $(0,3\epsilon)$. The time $t$ is such that
\[
t^{2}-4(s^{2}+h^{2})=36\epsilon^{2}.
\]
Hence
\[
t-2\sqrt{s^{2}+h^{2}}=36\epsilon^{2}/(t+2\sqrt{s^{2}+h^{2}}).
\]
Since $t>0$ and $s>0$, we have
\[
t-2\sqrt{s^{2}+h^{2}}<18\epsilon^{2}/h.
\]
This explains the factor 18 in Lemma \ref{ImagingSection:Lemma F0}. Now choosing $20$ (any number bigger than $18$ would do) explains our choice of the constant in \eqref{The function g(s,t)}.

\bibliographystyle{plain}
\def\dbar{\leavevmode\hbox to 0pt{\hskip.2ex \accent"16\hss}d}

\end{document}